\documentclass{article}

\usepackage{arxiv}

\usepackage[utf8]{inputenc} 
\usepackage[T1]{fontenc}    
\usepackage{url}            
\usepackage{booktabs}       
\usepackage{amsfonts}       
\usepackage{nicefrac}       
\usepackage{microtype}      
\usepackage{lipsum}

\usepackage{relsize,balance,lipsum,bbm,enumerate,times,comment,color,graphicx,setspace,mathdots,mathrsfs,amssymb,latexsym,amsfonts,amsmath,cite,stmaryrd,caption,pgf,accents,mathtools,tabu,enumitem,hhline,array,epstopdf,nicefrac,amsthm,microtype,algorithmic,array,float,bm,url}

\usepackage{graphicx}

\usepackage{mathtools} 

\usepackage[utf8]{inputenc}
\usepackage[english]{babel}
\DeclareMathOperator*{\argmax}{arg\,max}
\DeclareMathOperator*{\argmin}{arg\,min}
\usepackage{tikz}
\usetikzlibrary{positioning}
\definecolor{mygreen}{RGB}{153,255,153}
\definecolor{myorange}{RGB}{255,178,102}
\definecolor{myred}{RGB}{255,153,153}
\definecolor{myblue}{RGB}{153,204,255}

\newcommand{\ts}{\textsuperscript}

\title{Partial Observability Approach for the Optimal Transparency Problem in Multi-agent Systems}

\author{
    Sadegh~Arefizadeh\\
    Dept. of Electrical and Computer Engineering\\
    Tarbiat Modares University\\
    Tehran, Iran\\
    \texttt{sadegh.arefizadeh@modares.ac.ir}
\And
    Sadjaad~Ozgoli\\
    Dept. of Electrical and Computer Engineering\\
    Tarbiat Modares University\\
    Tehran, Iran\\
    \texttt{ozgoli@modares.ac.ir}
\And
    Sadegh~Bolouki\\
    Dept. of Electrical and Computer Engineering\\
    Tarbiat Modares University\\
    Tehran, Iran\\
    \texttt{bolouki@modares.ac.ir}
\And
    Tamer~Ba\c{s}ar\\
    Dept. of Electrical and Computer Engineering\\
    University of Illinois at Urbana-Champaign\\
    Urbana, IL, USA\\
    \texttt{basar1@illinois.edu}
}

\theoremstyle{definition}

\newtheorem{theorem}{Theorem}

\newtheorem{assumption}{Assumption}

\begin{document}
\maketitle




\begin{abstract}                          
    This paper considers a network of agents, where each agent is assumed to take actions optimally with respect to a predefined payoff function involving the latest actions of the agent's neighbors. Neighborhood relationships stem from payoff functions rather than actual communication channels between the agents. A principal is tasked to optimize the network's performance by controlling the information available to each agent with regard to other agents' latest actions. The information control by the principal is done via a partial observability approach, which comprises a static partitioning of agents into blocks and making the mean of agents' latest actions within each block publicly available. While the problem setup is general in terms of the payoff functions and the network's performance metric, this paper has a narrower focus to illuminate the problem and how it can be addressed in practice. In particular, the performance metric is assumed to be a function of the steady-state behavior of the agents. After conducting a comprehensive steady-state analysis of the network, an efficient algorithm finding optimal partitions with respect to various performance metrics is presented and validated via numerical examples.
\end{abstract}

\section{Introduction}

    Decision-making of self-interested interconnected agents is an interesting topic that has come up in various fields of research such as game theory \cite{27}, network science \cite{28}, and economics \cite{29}. In decision-making processes, information available to agents plays a crucial role as it influences the outcome of the process. From a game-theoretical perspective, there are two types of information, the availability of which, or lack thereof, have been discussed explicitly. The first type of information involves the game's defining parameters such as the players' action/strategy sets, their payoff functions, and rules of the game. When the entire information of this type is available to all players, the game is said to be one of {\it complete information}. Otherwise, a game of {\it incomplete information}, also known as a {\it Bayesian game}, ensues \cite{31}. The second type of information, which is of on-line type, and is exclusively defined for sequential games, involves the players' latest taken actions. The global availability of such information results in a game with {\it perfect information}, as opposed to one with {\it imperfect information} \cite{24}.

    In many cases, such as \cite{19,20}, incomplete and/or imperfect information arise from the nature of the problem, and therefore can not be regulated. However, there also exist situations where a principal has some control over the availability of information to the players. The extent to which information of aforementioned types is available to each player, herein broadly referred to as the level of {\it transparency}, influences the outcome of the game. Therefore, assuming that there exists a truthful principal with control over transparency, an optimal control problem emerges, with the objective of achieving an optimal network performance. Notable examples aiming to address this problem include \cite{26}, where the principal determines the set of players receiving the information signal, and 
    \cite{1}, which explores what portion of the population should receive the information signal to achieve optimality.


    In this work, inspired by network aggregative game models, we introduce and investigate a partial observability framework for the optimal transparency problem. Aggregative games are those where each player's payoff function depends on his own strategy and the aggregate of all other players' strategies. A network aggregative game (NAG) is a generalization of aggregative games in which each player's payoff is a function of his own strategy and the aggregation of his neighbors' strategies \cite{32}. Network aggregative games have been been utilized to model networks in different applications, including economics networks \cite{4}, opinion dynamics \cite{5}, and traffic networks \cite{6}. 
    The partial observability approach in this work involves a setting where (i) a principal is assumed to have control over the transparency of players' latest actions in the sense that she partitions the set of players into blocks and announces the mean action of each block publicly, and (ii) the players behave according to the best response strategy given the publicly available information. We address the optimal partition problem for two different network performance metrics, namely social welfare and free riding. We further take the players' privacy into consideration as a constraint in these optimization problems.


\subsection{Related Work}

    A large body of the game theory literature have been focused on the role of information in networks of interactive agents, with a few of remarkable ones described in the following. In \cite{19}, an algorithm for computing the Nash equilibrium is discussed for the case where players have asymmetric information about each other. The authors in \cite{20} investigate how information restriction associated with the players' degree of locality affects the achievable performance guarantees. In \cite{21}, a differential game is considered in which the players have incomplete information about their payoff functions and continually update their beliefs about the environment in order to improve their estimation of the real setting of the environment and optimize their payoffs. A special class of submodular optimization problems is addressed in \cite{22}, where each player has limited information about other players' actions and the effects of these limitations on the outcome are investigated. In \cite{23}, a mechanism is designed to maximize the welfare while satisfying some informational constraints.

    Moreover, there exist numerous applications in which the information transparency in a network is regulated by a central authority or principal. In these scenarios, the principal would seek to manage information transparency in such a way to induce an optimal outcome with respect to a performance metric of interest, which varies depending on the application. The means by which the principal controls the transparency can also change depending on the setup. For instance, the principal may be able to dictate who receives the information being disseminated \cite{26}, or what portion of the players should receive information \cite{1}.
    
    Regarding the performance metric by which the outcome is measured, the social welfare, defined as the aggregated payoffs of all players, is the most well-studied. In \cite{30}, it is argued that an intermediate level of transparency is optimal. In \cite{2,3}, a principal in a two-stage game aims to maximize the social welfare by determining the precision of public and private information available to agents. Another objective which would be sought by the principal is to eliminate or minimize free riding, which refers to the state of inaction by a player who takes advantage of other players' strategies and benefits from services or resources without paying for it. To the best of our knowledge, the free riding problem has not been discussed in the context of transparency in the literature. However, there are different other approaches to address the free riding problem. Some remarkable examples include having an assurance contract \cite{10}, in which the participants together pledge to a contract that force them to contribute for the public good; providing subsidies to motivate participants to contribute \cite{11}; and imposing a form of taxation to penalize free riders \cite{12}. In addition to social welfare and free riding, the notion of privacy of the players is also of great importance in practice. As an example, in the security investment decision-making processes, the information about investments of a firm is to remain private, while a state of absolute privacy is not necessarily optimal \cite{8}. This trade-off results in an optimal transparency problem discussed in \cite{9}.

    Two other popular frameworks concerned with the concept of information disclosure are signalling games and cheap talk, where there are two types of players, so-called {\it sender} and {\it receiver} \cite{akyol:16,sayin:19}. In these settings, the sender has some private information and sends an information signal to the receiver, based on which the receiver takes action. The payoffs of both player types are determined by the receiver's action. Due to the conflicting preference of the two player types, the sender would try to persuade the receiver to take action in such a way that maximizes his payoff, while the receiver would aim to find true information in order to maximize his own payoff. In \cite{akyol2015privacy}, privacy concerns of the players in such an environment are also investigated. 

\subsection{Contributions}

    In most of the existing work on optimizing the performance of a network, e.g., maximizing the welfare or minimizing free riding, a government/principal intervention targeting payoff functions of the players is required, for which managing the information transparency is a fine alternative. The body of literature on the transparency management problem largely involves manipulating the accuracy of information, which amounts to some level of disinformation. In this paper, we introduce and investigate a transparency management paradigm, the so-called partial observability approach, where the principal remains truthful throughout while the players' payoff functions are kept intact. Our main contributions in this paper are as follows.

    \begin{enumerate}
	\item We introduce the partial observability approach for managing information transparency where the information becomes ambiguous but remains veracious. It is done by a principal through partitioning agents into blocks and announcing only the mean of latest actions of players in each block instead of announcing all latest actions. The problem then becomes one of finding the optimal partition with respect to a performance metric of interest. We should note that this approach is inspired by the authors' previous work \cite{15}, where only the mean of {\it all} players' latest taken actions are made available in every stage and the limiting behavior of the ensuing best response dynamics is investigated.
    
    \item Focusing on a security investment decision-making process and assuming that agents adopt best response strategies given the available information, we show that for any partition employed by the principal, agents' actions asymptotically converge to individual limits. These limits are shown to be unique in the sense that they are independent of the initial conditions.
    
    \item We present an algorithm finding an optimal or nearly-optimal partition with respect to various performance metrics. In doing so, noticing that the number of partitions of a set is superexponential to the cardinality of the set,
    we convert the optimization problem in the discrete space of partitions to a continuous space convex optimization problem which in turn can be solved in polynomial time. We then use a community detection algorithm to find the partition that is closest to the solution in the continuous space.
    

\end{enumerate}
\subsection{Paper Organization}

    The remainder of the paper is organized as follows. In section \ref{section:2}, we present preliminaries and notations that we shall use later on in the paper. The formulation of the problem including the game model is detailed in Section \ref{section:3}. Then, in Section \ref{section:4}, we carry out an analysis of the game, with a particular focus on the convergence of its best response dynamics. Having analyzed the game, we address the main problem of improving the network's performance via a proper partitioning by the principal in Section \ref{section:5}. The effectiveness of our solution to this main problem is examined in Section \ref{section:6} via numerical examples. Finally, Section \ref{conclusion} concludes the paper with remarks and possible directions for future research.

\section{Notions and Terminology} \label{section:2}
    
    In this section, we state basic notions along with the terminology and notation which we shall use in the rest of the paper.

\subsection{Set Theory}
    
    Given a non-empty finite set $V$, its cardinality, that is the number of its members, is denoted by $|V|$. A {\it partition} of $V$ is a collection of non-empty pairwise-disjoint subsets of $V$, whose union is equal to $V$. Each set within a partition is referred to as a {\it block}. The number of all possible partitions of a set $V$, with $|V|=n$, is the well-known Bell number $B_n$, which according to \cite{17} satisfies the following inequalities for sufficiently large $n$:
    \begin{equation}\label{No. partitions}
        \left(\frac{n}{e \ln n}\right)^n< B_n < \left(\frac{n}{e^{1-\epsilon} \ln n}\right)^n.
    \end{equation}
    The lower bound on $B_n$ in \eqref{No. partitions} shows that the number of partitions grows superexponentially fast with $n$.

\subsection{Graph Theory}
    
    A directed, weighted graph is represented by $G(V,E,\varphi)$ where $V=\{1,\ldots,n\}$ is the set of nodes, $E \subset V \times V$ is the set of edges, and $\varphi:E\rightarrow \mathbb{R}$ determines the weight of each edge. A generic edge from node $i$ to node $j$ and its weight are denoted by $e_{ij}$ and $\varphi_{ij}$, respectively. We define the out-neighbor set, or simply the neighbor set, of node $i$ as
    \begin{equation}
        N_i=\{ j \in V \,|\, e_{ij} \in E\}.
    \end{equation}

\subsection{Convex Optimization}
    
    Convex optimization is a special case of mathematical optimization that takes the following form:
    \begin{equation}
        \begin{array}{cc}
            \text{minimize}
            & \hspace{-1in} f(x)\vspace{.05in}\\
            \text{subject to}& g_i(x) \geq 0, ~ i = 1,\ldots,m \vspace{.05in}\\
            & h_i(x) = 0,~ i =1, \ldots, p
        \end{array}
    \end{equation}
    where the objective function $f$ is a convex function and the feasible set, formed by the points satisfying the constraints, is a convex set. The convergence of algorithms solving convex optimization problems often requires a so-called {\it self-concordant} assumption. Self-concordant functions are those satisfying the following condition for any $x$:
    \begin{equation}\label{eq:20}
        |f^{'''}(x)| \leq 2f^{''}(x)^{\frac{3}{2}},
    \end{equation}
    where $f^{''}$ and $f^{'''}(x)$ are the second-order and third-order derivatives of $f$, respectively, which are assumed to exist. It is well-known that a large class of convex optimization problems are solvable in polynomial time \cite{16}.

\subsection{Notation}
    
    Given a matrix $W$, $W_i$ is the $i\ts{th}$ row and $W_{ij}$ is the element at the intersection of the $i\ts{th}$ row and $j\ts{th}$ column. The notation $W \geq 0$ is to be understood element-wise, while $W \succeq 0$ means that $W$ is positive-semidefinite. For a vector $v \in \mathbb{R}^{n}$, the operator $\max(0,v)$ is assumed to operate element-wise. $diag(v)$ is a diagonal matrix with elements of $v$ as its diagonal elements. $I$ and $\mathbf{1}$ denote the identity matrix and the vector of all ones, respectively, of proper orders. The norm $\|.\|$ indicates the $\infty$-norm, while the superscript $T$ on a vector or matrix stands for the transposition operation. 

\section{Problem Setup}\label{section:3}
    
    In this section, we describe the game model that we consider in this work as well as the partial observability approach and the formulation of the optimal transparency problem in detail.

\subsection{The Game Model}\label{the game model}

    The game model we focus on is the network aggregative game of \cite{13} pertaining to the security investment decision-making process of interconnected firms. In this model, $ V=\{1,\ldots,n\}$ is the set of players, $x_i$, $i \in V$, that has to be non-negative, denotes the action of each player, and $U_i (x_i,x_{-i})$ is its payoff function defined as


    \begin{equation} \label{eq:1}
        U_i (x_i,x_{-i}) = {S_i} (W_ix ) -c_ix_i,
    \end{equation}
    where $c_i$ is a positive constant, $W$ is a fixed $n \times n$ matrix whose diagonal elements are all equal to 1, and ${S_i}$ is a strictly concave, strictly increasing function which has the following properties:
    \begin{equation}
        {S_i}'(0) > c_i,~~ \lim_{x \rightarrow \infty} {S_i}'(x) < c_i.  
    \end{equation}
    One could interpret $W_ix$ as the effective investment of player $i$, ${S_i} (W_ix)$ as its pure payoff, and $c_i x_i$ as the expense of its investment $x_i$. The interested reader is referred to \cite{13} for a detailed interpretation of the payoff functions.

    Matrix $W-I$, whose diagonal elements are all zero, can be viewed as the weighted adjacency matrix of a graph $G(V,E,\varphi)$, which in essence indicates the interconnections among the players. This game is assumed to be played repeatedly and infinitely often. At each stage of the game, in the fully transparent setting, a principal would announce the latest actions of all players publicly. Then, each player, adopting the best response strategy, tries to maximize its payoff given the publicly available information, leading to the following best response dynamics \cite{13}:
    \begin{equation}
        x_i(t+1) = \max(0,({I_i}-{W_i})x(t)+b_i).
    \end{equation}
    or the equivalent vector form \cite{14}
    \begin{equation} \label{eq:2}
        x(t+1)=\max(0,(I-W)x(t)+b),
    \end{equation}
    where $b_i$ is the unique value which satisfies $S'(b_i) = c_i$, meaning that it is the optimal investment amount of player $i$ in the absence of any other player (that is, in a one-player game). It is shown in \cite{13} that if $W$ is diagonally dominant, that is if
    \begin{equation}\label{eq:30}
        \sum_{j\in N_i} w_{ij} < w_{jj} = 1,~ \forall i \in V,
    \end{equation}
    then the best response dynamics \eqref{eq:2} will asymptotically converge to the unique Nash-equilibrium of the game. Similar to \cite{13}, we also make the diagonal dominance assumption throughout the paper.
    \begin{assumption}\label{diagonal dominance}
        The matrix $W$ is diagonally dominant, which means that it satisfies \eqref{eq:30}. Equivalently, $I-W$ satisfies
        \begin{equation}\label{I-W norm}
            \|I-W\| < 1.
        \end{equation}
    \end{assumption}

\subsection{Partial Observability Approach}

    We present a framework for information announcement that does not taint its veracity, is public, and is capable of addressing the trade-off between gains from transparency and privacy protection. This framework involves a trusted principal aware of every player's latest action $x_i(t)$, $i \in V$. The principal partitions the players into a number of fixed blocks,
    \begin{equation}
        V_1,\ldots,V_m,
    \end{equation}
    and, at each stage of the game, calculates the average of latest taken actions in each block $V_k$, that is
    \begin{equation}
        \frac{1}{|V_k|}\sum_{i\in V_k} x_i,~k=1,\ldots,m,
    \end{equation}
    and announces it publicly.

\subsection{Best Response Dynamics Under Partial Observability}
     Let $p$ denote the partition employed by the principal and $I_p(x)$ denote the information announced by the principal with regard to $x$ given partition $p$. As stated previously, each player is assumed to adopt the best response strategy, which in the non-fully transparent case leads to the following best response dynamics for the game:
    \begin{equation}\label{BR}
        x(t+1)=BR(x(t)|I_p(x(t)),
    \end{equation}
    where $BR$ indicates the adopted best response strategy. In order for players to best respond to the state of the game given the available information $I_p(x(t))$, they first estimate other players' actions at time $t$, in particular those of their neighbors. Noticing that the available information with regard to $x_i(t)$ is the mean value of actions at time $t$ in the block containing $i$, we reasonably assume in this work that the players estimate $x_i(t)$ with that mean value. Thus, partition $p$ employed by the principal directly influences the estimates of players' actions, and consequently, the best response dynamics and the evolution of actions depend on the partition $p$. We will delve deeper into this modified best response dynamics and its formulation in Section \ref{section:4}.
    

\subsection{Problem Definition} \label{problem definition}

    Having introduced the framework for information announcements and the best response dynamics, one has to now address the convergence properties of the resulting best response dynamics of the game. More precisely, given an arbitrary partition $p$, will the best response dynamics converge to an equilibrium and is this equilibrium unique in the sense that it is independent of the initial conditions? In other words, will the players' actions evolving under the best response dynamics asymptotically converge to individual limits, and if so, are they independent of the initial actions? One further wonders how these limits are characterized, and in particular, how they vary with respect to the partition $p$.

    A more important problem to address is that which of the possible partitions is optimal with respect to a given performance metric for the network. More precisely, assuming that for each partition $p$, there is a unique equilibrium $x^*_p$ for the best response dynamics of the game, and a performance metric is given as a function of $x^*_p$, say $J(x^*_p)$, one seeks a partition $p$ optimizing the performance, that is
    \begin{equation}
        \argmax_{p\in P}~J(x^*_p),
    \end{equation}
    where $P$ denotes the set of all partitions. For simplicity, we may drop the subscript $p$ and use $x^*$ for the equilibrium when no ambiguity results. Two different performance 
    metrics are of particular interest in this work. One metric is the social welfare defined as
    \begin{equation}
        \sum_{i \in V} U_i(x^*_i,x^*_{-i}),
    \end{equation}
    and another one is (the negation of) the aggregated free riding index of all or some of the players. The free riding index of a player $i$ in \cite{13} is defined as
    \begin{equation}\label{free riding index}
        \gamma_i(x^*)=\frac{W_i x^* -x^*_i}{b_i}.
    \end{equation}
    Remembering that $b_i$ can be interpreted as the player $i$'s default investment, that is if the player was isolated from the network, and $W_i x^* -x^*_i$ is the effective investment amount gained by the player for being in the network, $\gamma_i(x^*)$ in \eqref{free riding index} very well captures how much of a free ride the player is taking, with smaller indices corresponding to less free riding.


    Finally, one can further elaborate and address the issue of privacy protection by adding a constraint to the optimization problem, that is imposing a lower bound on the cardinality of each block in the partition $p$. This constraint is inspired by the intuition that larger blocks lead to more ambiguous information made public and are therefore more protective of their members' private information.
    

\section{Best Response Dynamics Analysis} \label{section:4}

    In this section, we characterize the best response dynamics \eqref{BR} of the game presented in Section \ref{section:3} with partial observability and investigate its convergence properties.


    Given the most recent information announced by the principal, which pertains to $x(t)$, players first estimate other players' latest actions, their neighbors' in particular, to then take actions at $t+1$ according to the best response strategy. Recalling that (i) each player belongs to exactly one block and (ii) the only public information available regarding that block is the mean of latest actions by its members, the player's latest taken action is estimated by the announced mean of its containing block.

    This estimation of $x(t)$ by a player solely relies on the information announced publicly. One notices that each player $i$ also possesses some private information, that is indeed the value of $x_i(t)$. However, we assume that the number of players is sufficiently large that the impact of this private information is negligible in estimating $x(t)$ as a whole. Consequently, a uniform estimate $\tilde{x}_p(t)$ of $x(t)$ is derived, where the subscript $p$ indicates the partition employed by the principal. By the ``uniform'' estimate, we mean that for every player $j$, $(\tilde{x}_p)_j(t)$ is the estimate of $x_j(t)$ by every player $i$, $i \neq j$.
    
    The vector $\tilde{x}_p(t)$ can be formulated as $H_px(t)$, where $H_p$ is a symmetric $n \times n$ matrix constructed from the partition $p$ as the following. Precisely, if $l_i$ denotes the size of the block containing player $i$, the matrix $H_p$ is constructed according to
    \begin{equation}\label{H_p}
        (H_p)_{ij} =
        \begin{cases}
            l_i^{-1} & \text{if $i$ and $j$ belong to the same block of $p$},\\
            0 & \text{otherwise.}
        \end{cases}
    \end{equation}
    Now, in view of equations \eqref{eq:1} and \eqref{eq:2}, one can write the following best response dynamics:
    \begin{equation}\label{eq:16}
        \begin{array}{ll}
            x(t+1)
            & \hspace{-.1in}=\max(0,(I-W)\tilde{x}_p(t)+b)\vspace{.05in}\\
            &
            \hspace{-.1in}= \max(0,(I-W)H_px(t)+b).
        \end{array}
    \end{equation}

    As it is evident from \eqref{eq:16}, the partition $p$ influences the evolution of players' actions. Thus, one aims to find an optimal partition $p$ with respect to a desired performance metric. As we are interested in performance metrics involving the steady-state behaviors of players in this work, we first show that for any partition $p$, each player's action converges as time grows.

    \begin{theorem}\label{theorem1}
        Given any partition $p$ employed by the principal and the best response dynamics \eqref{eq:16}, $\lim_{t \rightarrow \infty} x(t)$ exists and is independent of the initial vector $x(0)$ of actions.
    \end{theorem}
    \begin{proof}
        From \eqref{eq:16}, one has
        \begin{equation}\label{eq:19}
            \begin{array}{l}
                \| x(t+2)-x(t+1)\|\vspace{.05in}\\
                ~~~~=\|\max(0,(I-W)H_px(t+1)+b)\vspace{.05in}\\
                ~~~~~~-\max(0,(I-W)H_px(t)+b)\|\vspace{.05in}\\
                ~~~~\leq \|((I-W)H_px(t+1)+b)-((I-W)H_px(t)+b)\|\vspace{.05in}\\
                ~~~~=\|((I-W)H_p)(x(t+1)-x(t))\|\vspace{.05in}\\
                ~~~~\leq \|I-W\| ~\|H_p\|~ \|(x(t+1)-x(t)\|\vspace{.05in}\\
                ~~~~=\gamma \|x(t+1)-x(t)\|,
            \end{array}
        \end{equation}
        where $\gamma=\|I-W\| < 1$ according to \eqref{I-W norm}, which proves the convergence of $x(t)$ by contraction. We also note that to write the first inequality in \eqref{eq:19}, we took advantage of the following inequality that holds for any $a,b \in \mathbb{R}$:
        \begin{equation}
            |\max(0,a) - \max(0,b)| \leq |a-b|,
        \end{equation}
        which can be easily proved by considering all four cases for the signs of $a$ and $b$. Furthermore, in the last equality in \eqref{eq:19}, we replaced $\|H_p\|$ by 1, since $H_p \geq 0$ and each row of $H_p$ sums up to 1.


        To show that the equilibrium point is independent of the initial vector $x(0)$ of actions, let $x^*_1$ and $x^*_2$ be two equilibria given the dynamics \eqref{eq:16}. Therefore, we must have
        \begin{equation}
            x^*_1=\max(0,(I-W)H_px^*_1+b),
        \end{equation}
        \begin{equation}
            x^*_2=\max(0,(I-W)H_px^*_2+b).
        \end{equation}
        Now starting from $\|x^*_1 - x^*_2\|$ and following the same lines of argument as in \eqref{eq:19}, we arrive at the following inequality:
        \begin{equation}
            \begin{array}{l}
                \|x^*_1-x^*_2\|\vspace{.05in}\leq \gamma \|x^*_1-x^*_2\|,
            \end{array}
        \end{equation}
        for $\gamma < 1$, which immediately results in $x^*_1=x^*_2$. This proves the uniqueness of the equilibrium given partition $p$ and the proof of the theorem is now complete.
    \end{proof}
    Having shown that there is a unique equilibrium for each partition, we next characterize this equilibrium more precisely.
    \begin{theorem}
        Let $p$ be the partition employed by the principal, and $H_p$ be constructed from $p$ according to \eqref{H_p}. Then, the unique equilibrium of the best response dynamics \eqref{eq:16} satisfies the following linear complementarity conditions:
        \begin{equation}\label{eq:6}
            \begin{cases}
                y=(I-(I-W)H_p)x^*-b\\
                y^Tx^*=0\\
                y\geq 0 , x^*\geq 0.\\
            \end{cases}
        \end{equation}
    \end{theorem}
    \begin{proof}
        If $x^*$ is the unique equilibrium of the best response dynamics \eqref{eq:16}, one must have
        \begin{equation}\label{eq:18}
            x^* = max(0,(I-W)H_px^*+b).
        \end{equation}
        Therefore,
        \begin{equation}
           \begin{cases}
                x^* \geq 0\\
                x^* \geq (I-W)H_px^*-b.\\
            \end{cases}   
        \end{equation}
        Thus, one can write
        \begin{equation}
            \begin{cases}
                x^*\geq0\\
                (I-(I-W)H_p)x^*-b\geq0\\
                y=x^*(I-(I-W)H_p)-b,
            \end{cases}   
        \end{equation}
        where $y \geq 0$. Moreover, from \eqref{eq:18}, for every $i$ we have
        \begin{equation}
            \begin{cases}
                x_i^*=0\vspace{.05in}\\
                \text{or}\vspace{.05in}\\
                y_i = \big((I-(I-W)H_p)x^*\big)_i-b_i=0,
            \end{cases}
        \end{equation}
        which results in $y^Tx^* = 0$. Hence, all of linear complementarity conditions \eqref{eq:6} are satisfied.
    \end{proof}
    We have shown thus far that the equilibrium of the best response dynamics \eqref{eq:16} exists, is unique, and satisfies the linear complementarity conditions \eqref{eq:6}. Recalling that our main objective is finding an optimal partition $p$ with respect to a given performance metric, which here is a function of the unique equilibrium of \eqref{eq:16}, we now present an approximation of the equilibrium that we shall use in the rest of the paper.

    Assuming $x^* > 0$, from \eqref{eq:6} we have $y = 0$, and therefore $x^*$ is obtained as
    \begin{equation}\label{eq:8}
        x^*=(I-(I-W)H_p)^{-1} b.
    \end{equation}
    Noticing that
    \begin{equation}
        \|(I-W)H_p\| \leq \|(I-W)\|\, \|H_p\| < 1, 
    \end{equation}
    we use the first two terms of the Neumann series for $(I-(I-W)H_p)^{-1}$ to approximate it, that is $I + (I-W)H_p$. Hence, in case $x^*> 0$, we have the following approximation for $x^*$:
    \begin{equation}\label{eq:7}
        x^* \approx (I+(I-W)H_p)b.
\end{equation}
One notices that for the approximation \eqref{eq:7} to be more accurate, $\|(I-W)H_p\|$ or $\|I-W\|$ should be small. 
Since \eqref{eq:7} is valid only if $x^* > 0$, we present a sufficient condition for $x^* > 0$, which is reasonably weak for small $\|I-W\|$.
\begin{theorem}\label{x>0}
For the unique equilibrium $x^*$ of dynamics \eqref{eq:16}, we have $x^* > 0$ if
\begin{equation}\label{b condition}
    \min_{j\in V}b_j
    >
    \frac{\|I-W\|}{1-\|I-W\|}\, \max_{j\in V} b_j.
\end{equation}
\end{theorem}
\begin{proof}
According to \eqref{eq:8} and the Neumann series, for any $i\in V$ we have
    \begin{align}
        x^*_i
        & =\big((I-(I-W)H_p)^{-1}\big)_ib\nonumber\\
        & = b_i + \sum_{m=1}^{\infty}\big((I-W)H_p\big)_i^mb\nonumber\\
        & \geq \min_{j\in V}b_j - \sum_{m=1}^{\infty} \|(I-W)\|^m\, \|H_p\|^m \max_{j\in V}b_j\nonumber\\
        & = \min_{j\in V}b_j - \frac{\|I-W\|}{1-\|I-W\|} \max_{j\in V}b_j\nonumber\\
        & > 0,
    \end{align}
where to write the last inequality, we took advantage of the condition \eqref{b condition}.
\end{proof}

\section{Finding an Optimal Partition}\label{section:5}

    Recalling \eqref{No. partitions}, the number of partitions of a set with cardinality $n$ is at least $\left(\frac{n}{e \ln n}\right)^n$, which makes the problem of finding the solution of     \begin{equation}
        \argmax_{p\in P}~J(x^*_p)
    \label{opt problem}
    \end{equation}
    very inefficient for large $n$ if an exhaustive search is conducted. In this section, for two performance metrics, we present algorithms finding an optimal or nearly-optimal partition in polynomial time. We shall convert the discrete optimization problem into a continuous one so that various continuous optimization techniques can be applied. After obtaining a solution in the continuous space, a community detection algorithm will be utilized to find the nearest point in discrete space of all partitions.
    
    We recall that any partition $p$ corresponds to a unique matrix $H_p$ according to \eqref{H_p}. It should be clear that $H_p$ is always doubly-stochastic, meaning that $H_p \geq 0$ and each row/column of $H_p$ sums up to 1. One also notices that $H_p$ is positive-semidefinite since it is symmetric and satisfies
    \begin{equation}
        H_p=H_p^2 = H_p^TH_p,
    \end{equation}
    which means that for any $v \in \mathbb{R}^n$,
    \begin{equation}
        v^T H_p v = v^TH_p^TH_pv = (H_pv)^TH_pv \geq 0.
    \end{equation}
    We now extend the discrete space $P$ of partitions, or their corresponding doubly-stochastic positive-semidefinite matrices, to the convex set of all doubly-stochastic positive-semidefinite matrices to solve the optimization problem \eqref{opt problem}.

    In the following two subsections, we consider two different optimization objectives and find an optimal solution in the aforementioned continuous space accordingly. Then, we address the issue of privacy by adding a constraint to the optimization problems and solving these modified problems. Finally, in the last subsection, we use a community detection algorithm to derive partitions corresponding to the optimal solutions obtained.  
    The optimization objectives that we consider are (i) maximizing social welfare and (ii) minimizing free riding done by a player or a group of players. A combination of these optimization objectives is as well of interest although it is not discussed in this paper.
    
\subsection{Social Welfare}

    The social welfare metric is defined as the aggregated payoff of all players when the unique equilibrium is reached
    , i.e.,
    \begin{equation}
        \text{welfare}=        \sum_{i \in V} U_i(x^*_i,x^*_{-i})=\sum_{i=1}^{n} {S_i}(W_ix^*)-c_ix^*_i.
    \end{equation}
    Based on \eqref{eq:7}, we use the approximated vector $(I+(I-W)H_p)b$ for $x^*$ to find a partition $p$, or equivalently $H_p$, that maximizes welfare. For now, as explained before, we assume that $H_p$ can be any doubly-stochastic positive-semidefinite matrix generically denoted by $H$ and aim to solve the following optimization problem:
    \begin{align}
        &\underset{H}{\argmax}
        ~ \text{welfare}(H)=\sum_{i=1}^{n}{S_i}(W_i(I+(I-W)H)b)\nonumber\\
        &\hspace{1.5in} -c_i(I+(I-W)H)_ib\nonumber\\
        &~~~~~~~~\text{subject to}~
        \begin{cases}
            H\succeq 0\vspace{.05in}\\
        H\,\mathbf{1}=\mathbf{1}\vspace{.05in}\\
        \mathbf{1}^TH = \mathbf{1}^T
        \end{cases}\label{eq:10}
    \end{align}
    The welfare function is concave with respect to $H$ since each player's payoff function is concave. Furthermore, the set of doubly-stochastic positive-semidefinite matrices is convex, which means that the feasible set is convex. Thus, \eqref{eq:10} is a convex optimization problem that can be solved in polynomial time. It should be noted that the optimal solution, denoted by $H^*$, does not immediately correspond to a partition $p$. We later on in Subsection \ref{section:5.4} make explicit how a partition $p$ can be obtained in such a way that $H_p$ is closest to $H^*$.

\subsection{Free Riding}

    Another important performance metric for the network is the amount of free riding by a player or a collection of players. A formulation of free riding is given in \eqref{free riding index} that is suitable for the fully transparent case. For the case with partial observability, we introduce the following free riding metric:
    \begin{equation}\label{eq:13}
        \eta_i=\frac{b_i  - x^*_i}{b_i},
    \end{equation}
    one notices that the metric \eqref{eq:13} becomes identical to \eqref{free riding index} in the fully observable case, which is when all blocks of the partition $p$ are singletons. Aiming to minimize the total amount of free riding by all players, the following convex optimization problem is to be addressed:
    \begin{align}
        &\underset{H}{\argmin}
        \sum_{i \in V}\eta_i = \mathbf{1}^T diag(b)^{-1}(b-(I+(I-W)H)b)\nonumber\\ 
        &~~~~~~~~\text{subject to}~
        \begin{cases}
            H\succeq 0 \vspace{.05in}\\
            H\,\mathbf{1}=\mathbf{1}\vspace{.05in}\\
                \mathbf{1}^TH = \mathbf{1}^T
        \end{cases}\label{eq:11}
    \end{align}
    We also point out that a similar convex optimization problem can be written for any subset of players whose total amount of free riding is sought to be minimized. Once again, we leave the part where we find the partition $p$ corresponding to the optimal solution $H^*$ of \eqref{eq:11} to Subsection \ref{section:5.4}.

\subsection{Privacy Constraint}\label{section:5.3}

    When performing the partitioning of the players, the principal could as well take the privacy of the players into consideration. To incorporate the notion of privacy into our problem formulation, we notice that larger partition blocks can be viewed as more protective of their members' privacy since the members' investment values become more obscure and more uncertain. Therefore, a certain degree of privacy can be translated to a lower bound on the cardinality of each block. Since this lower bound constraint shrinks the feasible set of partitions, it can worsen the achievable performance of the network. Hence, there is indeed a trade-off between protecting privacy and the achievable performance.


    The lower bound constraint on the cardinality of each block is in regard to the discrete optimization problems. Therefore, one has to now translate this constraint to one on the matrix $H$ in the space of doubly-stochastic positive-semidefinite matrices. Noticing that a partition $p$ with a lower bound $L$ on the cardinality of its blocks corresponds to $H_p$ all of whose elements are at most $1/L$, the constraint $H \leq 1/L$, which is to be understood element-wise, is added to the optimization problems \eqref{eq:10} and \eqref{eq:11}. It should be clear that the feasible set remains convex with the added constraint. The resulting convex optimization problems are
    \begin{align}
            &\underset{H}{\argmax}
            ~ \text{welfare}(H)=
            \sum_{i=1}^{n}{S_i}(W_i(I+(I-W)H)b)\nonumber\\
            & \hspace{1.5in}-c_i(I+(I-W)H)_ib\nonumber\\
        &~~~~~~~~\text{subject to}~
        \begin{cases}
            H\succeq 0 \vspace{.05in}\\
            H \leq 1/L\vspace{.05in}\\
            H\,\mathbf{1}=H^T\,\mathbf{1}=\mathbf{1}
            \end{cases}
        \label{eq:10 with privacy}
    \end{align}
and
    \begin{align}
        &\underset{H}{\argmin}\sum_{i \in V}\eta_i=\mathbf{1}^T diag(b)^{-1}(b-(I+(I-W)H)b)\nonumber\\ 
        &~~~~~~~~\text{subject to}~\begin{cases}
                H\succeq 0 \vspace{.05in}\\
                H \leq 1/L\vspace{.05in}\\
                H\,\mathbf{1}=H^T\,\mathbf{1}=\mathbf{1}
            \end{cases}\label{eq:11 with privacy}
    \end{align}
\subsection{From the Optimal Solution $H^*$ to a Partition} \label{section:5.4}

    Having obtained the solution $H^*$ of either of the optimization problems \eqref{eq:10}, \eqref{eq:11}, \eqref{eq:10 with privacy}, or \eqref{eq:11 with privacy}, a partition $p$ for which $H_p$ is closest to $H^*$ is now desired. To find $p$, we perform a community detection algorithm proposed in \cite{18} on the graph whose adjacency matrix is $H^*$. This graph is undirected since $H^*$ is symmetric. In \cite{18}, a function is introduced that quantifies the fitness of any given cluster $\varsigma$ in the graph, i.e.,
    \begin{equation}
        f_{\varsigma}=\frac{k_{in}^{\varsigma}}{(k_{in}^{\varsigma}+k_{out}^{\varsigma})^{\alpha}},
    \end{equation}
    where $k_{in}^{\varsigma}$ is the total in-cluster degrees, which is the summation of weights of all edges with both ends inside $\varsigma$. Moreover, $k_{out}^{\varsigma}$ is the total cross-cluster degrees, which is the summation of weights of all edges with one end inside $\varsigma$ and one end outside it. We use the algorithm presented in \cite{18} with $\alpha=1$ to find a partition of nodes, blocks of which have the highest total fitness. It should be noted that the complexity of this algorithm is $O(n^2)$.

\section{Numerical Examples}\label{section:6}
    
    In this section, we provide various examples to demonstrate the results from the previous sections via simulations. In the first example, for a network of 10 agents, we find through an exhaustive search the partition that minimizes the total free riding. Then, we utilize the algorithm from the previous section to find a partition among all possible partitions expected to minimize the total free riding. We will see that the two obtained partitions are indeed identical, which confirms the effectiveness of our algorithm. In the second example, we repeat the steps of the first example for the same network, however in the presence of privacy constraints. In the final example, we consider a network of 50 agents, for which an exhaustive search to find an optimal partition is computationally too complex, to show that our algorithm can still find the desired partition in reasonable time. We assume throughout this section that
    \begin{equation}\label{utility-function}
       {S_i}(x)=200\sqrt{x},~\forall i \in V,
    \end{equation}
    which satisfies the conditions for the payoff functions given in Subsection \ref{the game model}.
    
    Given a network of 10 agents, let us generate a non-negative vector $b$ and a diagonally dominant matrix $W$, with 1 as all its diagonal elements, randomly in such a way to satisfy condition \eqref{b condition}:
    \begin{equation*}
        b=[773 \ \ 719  \ \ 152  \ \ 498  \ \ 990  \ \  170  \ \ 328 \ 727 \ \  621  \ \  627]^T, 
    \end{equation*}
    \begin{equation*}
        W=\begin{bmatrix}
            1 & .07 & .11&0  & .04 & 0 &.07  &.08  &.01&.08  \\
            .02 & 1 & .04&0  & .06 & .04 &0  &.01  &0&0\\
            0 & 0 & 1&0  & 0 & 0 &0  &0  &.01&.25\\
            0 & 0 & 0&1  & .06 & .02 &.18  &.03  & 0&.19\\
            0 & .07 & 0&.18  & 1 & 0 &0  &.05  &.01&.17 \\
            .04 &.03 & .01&.11  & 0 &1 &0  &0  &0&.18\\
            .02 & .11 & .06&0  & .08 &.05 &1  &0  &.07&.04\\
            0 & .11 &.10&.02  & 0 & 0 &.09  &1  &.06&.08\\
            0 & 0& .13& .01 &0  &.12  &0&.06&1&.04\\
            0 &0 & .17&0  & 0 &.07 &.16  &0  &.09&1\\
        \end{bmatrix}.
    \end{equation*}
    Carrying out an exhaustive search in the set of all partitions, we conclude that the partition minimizing the total free riding is
    \begin{equation}\label{opt part}
        \{1,5,8 \},\{2,9 \},\{ 3,6,7,10\},\{4 \}.
    \end{equation}
    Fig.~\ref{fig:Dynamics} shows the evolution of players' actions given that this partition is adopted by the principal. It can be seen in Fig.~\ref{fig:Dynamics} that every player's action asymptotically converges as time grows, which is consistent with the result of Theorem~\ref{theorem1}.

    \begin{figure}[t]
    \centering
        \includegraphics[width=.5\linewidth]{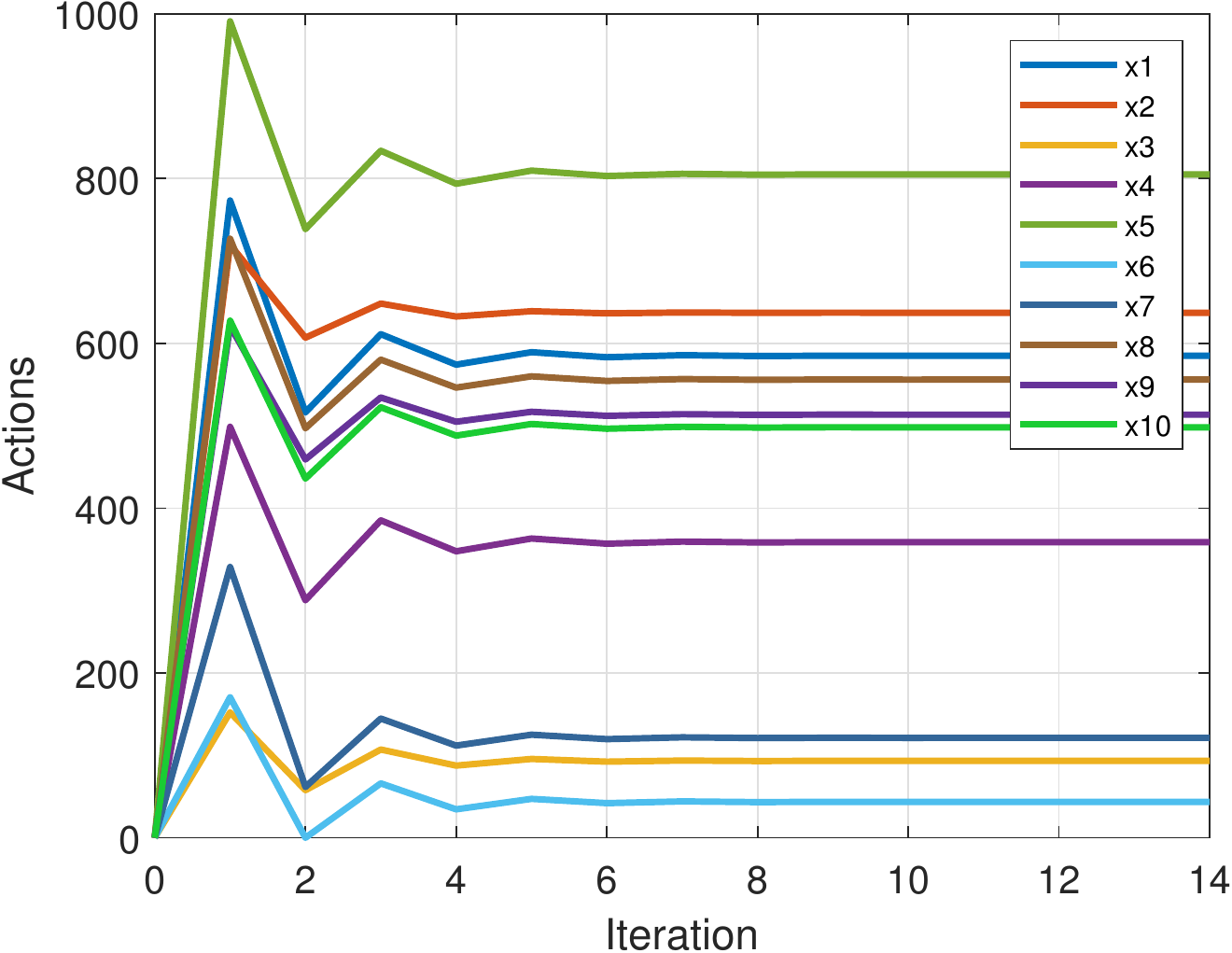}
        \caption{Players' actions over time given the optimal partition.}
    \label{fig:Dynamics}
    \end{figure}

    We now simulate the algorithm presented in Section \ref{section:5} to find a partition that is expected to minimize the total free riding. Matrix $H^*$ is derived by solving the optimization problem \eqref{eq:10}, resulting in the graph drawn in Fig.~\ref{fig:2}, the adjacency matrix of which is equal to $H^*$. One should notice that the thickness of edges in the graph indicates their weights.
    
    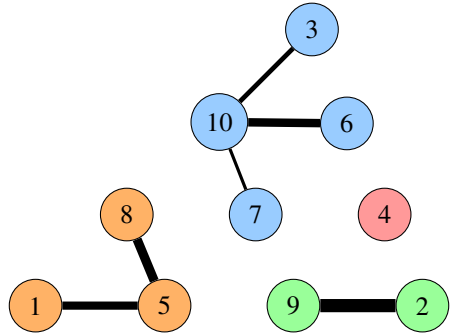
\begin{figure}\centering
        \begin{tikzpicture}[
node distance=1cm,
mynode/.style={
  circle,
  draw,
  fill=mygreen,
  align=center,minimum size=7mm}
]

\node[mynode,fill=myorange](1){1} ;
\node[mynode,fill=myorange](5)[right=of 1]{5} ;  
\node[mynode,fill=myorange](8)[above right=of 1]{8};
\node[mynode,fill=myblue](7)[right=of 8]{7};
\node[mynode,fill=myblue](10)[above right=of 8]{10};
\node[mynode](9)[right=of 5]{9};
\node[mynode](2)[right=of 9]{2};
\node[mynode,fill=myblue](6)[above right=of 7]{6};
\node[mynode,fill=myblue](3)[above right=of 10]{3};
\node[mynode,fill=myred](4)[above right =of 9]{4};
\draw[-,line width=3.1pt,] 
  (1) -- node[rotate=90,below,] {} (5); 
  \draw[-,line width=5pt,] 
  (2) -- node[rotate=90,below,] {} (9); 
  \draw[-,line width=1.8pt,] 
  (3) -- node[rotate=2.7,below,] {} (10);
  \draw[-,line width=3.5pt,] 
  (5) -- node[rotate=90,below,] {} (8); 
  \draw[-,line width=3.2pt,] 
  (6) -- node[rotate=90,below,] {} (10); 
  \draw[-,line width=1.2pt,] 
  (7) -- node[rotate=90,below,] {} (10);
  
\end{tikzpicture} 
 \caption{Graph with adjacency matrix $H^*$} \label{fig:2}
 \end{figure}
    
    One clearly observes that communities are formed between nodes of the same color, a result that is made concrete by applying the community detection algorithm of Subsection \ref{section:5.4} on the graph with adjacency matrix $H^*$. These communities constitute the same partition as that of \eqref{opt part}, which shows the effectiveness of our process of finding an optimal partition.
    
    Next, we impose a lower bound value of 3 on the cardinality of all blocks in the partition, meaning that no singleton block or block of size 2 is permitted, while the objective is to minimize the total free riding. Performing an exhaustive search in the feasible set of partitions results in the following optimal partition:
    
    \begin{equation}\label{opt part 2}
        \{1,5,8 \},\{2,4,9 \},\{3,6,7,10\}.
    \end{equation}
    Now, implementing our algorithm proposed in Section \ref{section:5}, we derive the matrix $H^*$ that is the solution of \eqref{eq:11 with privacy}, and draw its corresponding graph in Fig.~\ref{fig:3}. We then apply the community detection algorithm of Subsection \ref{section:5.4} on this graph, which leads to the same partition as in \eqref{opt part 2}, confirming the effectiveness of our algorithm in the case where the privacy constraint is imposed.

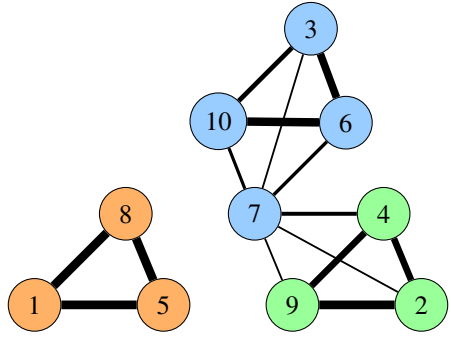
\begin{figure}[t]\centering
    \begin{tikzpicture}[
node distance=1cm,
mynode/.style={
  circle,
  draw,
  fill=mygreen,
  align=center,minimum size=7mm}
]

\node[mynode,fill=myorange](1){1} ;
\node[mynode,fill=myorange](5)[right=of 1]{5} ;  
\node[mynode,fill=myorange](8)[above right=of 1]{8};
\node[mynode,fill=myblue](7)[right=of 8]{7};
\node[mynode,fill=myblue](10)[above right=of 8]{10};
\node[mynode](9)[right=of 5]{9};
\node[mynode](2)[right=of 9]{2};
\node[mynode,fill=myblue](6)[above right=of 7]{6};
\node[mynode,fill=myblue](3)[above right=of 10]{3};
\node[mynode,fill=mygreen](4)[above right =of 9]{4};
\draw[-,line width=3.3pt,] 
  (1) -- node[rotate=90,below,] {} (5); 
  \draw[-,line width=3.3pt,] 
  (1) -- node[rotate=90,below,] {} (8);

  \draw[-,line width=3.3pt,] 
  (2) -- node[rotate=90,below,] {} (9); 
    \draw[-,line width=2.5pt,] 
  (2) -- node[rotate=90,below,] {} (4); 
    \draw[-,line width=0.7pt,] 
  (2) -- node[rotate=90,below,] {} (7); 
  \draw[-,line width=1.8pt,] 
  (3) -- node[rotate=2.7,below,] {} (10);
  \draw[-,line width=2.9pt,] 
  (3) -- node[rotate=2.7,below,] {} (6);
    \draw[-,line width=0.7pt,] 
  (3) -- node[rotate=2.7,below,] {} (7);
    \draw[-,line width=1.5pt,] 
  (4) -- node[rotate=2.7,below,] {} (7);
    \draw[-,line width=2.5pt,] 
  (4) -- node[rotate=2.7,below,] {} (9);
    \draw[-,line width=3.3pt,] 
  (5) -- node[rotate=90,below,] {} (8); 
    \draw[-,line width=1.3pt,] 
  (6) -- node[rotate=90,below,] {} (7); 
    \draw[-,line width=3.2pt,] 
  (6) -- node[rotate=90,below,] {} (10);
   \draw[-,line width=0.7pt,] 
  (7) -- node[rotate=90,below,] {} (9);
  \draw[-,line width=1.2pt,] 
  (7) -- node[rotate=90,below,] {} (10);

\end{tikzpicture} 
\caption{Graph with adjacency matrix $H^*$ given the privacy constraint} \label{fig:3}
\end{figure}  
    
    As we mentioned earlier, the main application of our process of finding an optimal partition is in the case of large networks, where the number of agents renders an exhaustive search among the partitions computationally too complex. Thus, we now consider a network of 50 agents and aim to find a partition that, when adopted by the principal, maximizes the social welfare. For simulation, we consider utility functions in the form of \eqref{utility-function} and generate matrix $W$ and vector $b$ randomly in such a way that it satisfies $\|(I-W)H_p\|<0.5$ and condition \eqref{b condition}. A lower bound value of 4 is also imposed on the cardinality of each block of the partition as the privacy constraint. We should note that in this case, we will not be able to validate the result of our algorithm by conducting an exhaustive search due to its enormous computational complexity.
    
    Initiating our process of finding an optimal partition, we derive the matrix $H^*$ by solving the convex optimization problem \eqref{eq:11 with privacy}. The graph whose adjacency matrix is $H^*$ is drawn in Fig.~\ref{fig:4}, where communities corresponding to the blocks of the partition are also detected and given the same color. We finally reiterate that our proposed algorithm of finding optimal partitions runs in polynomial time, while an exhaustive search has a superexponential computational complexity.

\begin{figure}[t]
    \centering
    \includegraphics[width=.6\linewidth]{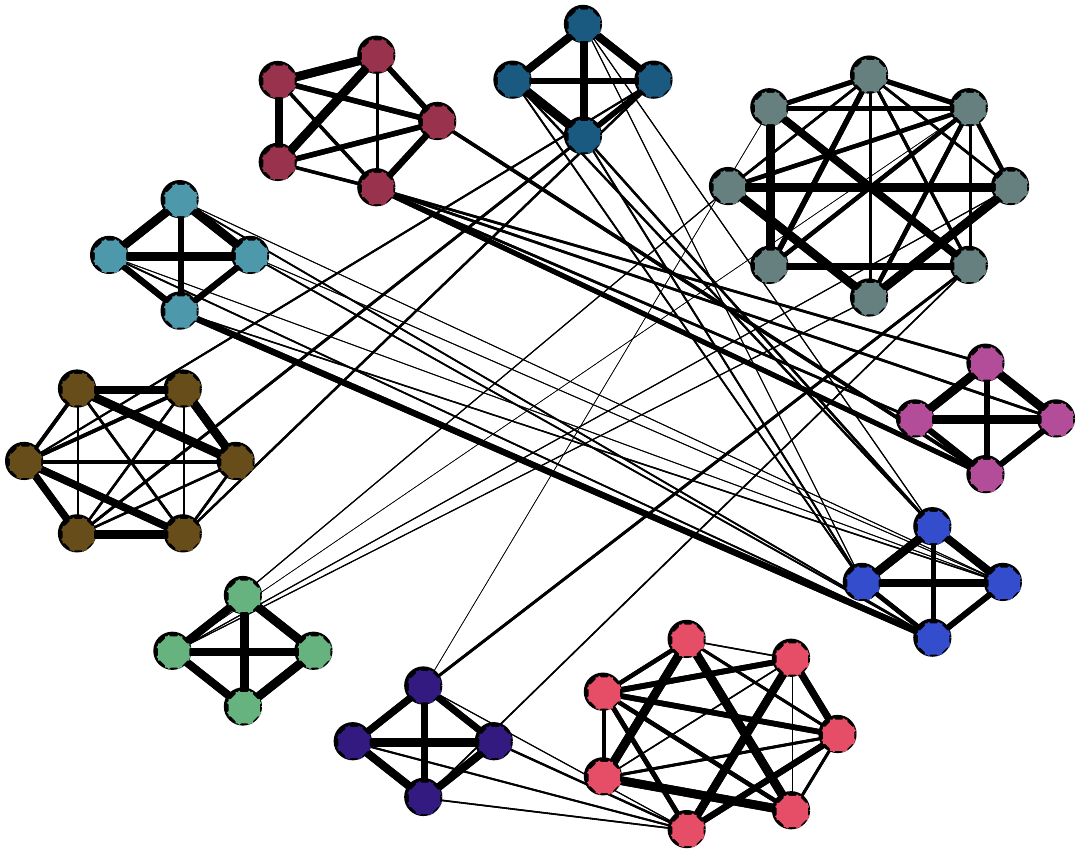}

     \centering
    \caption{Graph with adjacency matrix $H^*$ for the network of 50 agents}
    \label{fig:4}
\end{figure}
\section{Concluding Remarks and Future Work}\label{conclusion}

    In this paper, we introduced the partial observability approach for improving the performance of a network. We assumed that a principal has the authority to partition the agents into blocks and announce the mean value of each block's members' latest actions publicly. This selective disclosure of information influences the evolution of agents' actions, and consequently, the performance of the network. We conducted a comprehensive analysis of the network dynamics under the partial observability approach, addressed how the principal finds optimal partitions with respect to various performance metrics, and validated our findings via numerical examples.
    
    We argued that the proposed information announcement paradigm is also capable of addressing the issue of privacy, as blocks of larger cardinalities are considered to be less revealing of their members' actions. Although the proposed approach can be applied to arbitrary game models and dynamical systems, we focused on a network aggregative game model for security decision-making of inter-connected agents.
    
    
    The generality of the original problem allows for multiple directions for future research. The method by which the players estimate each others' actions at each stage of the game can be generalized to take into account a richer history of the game than the single previous stage. Furthermore, dynamic partitioning of the players is of great interest. Finally, a realistic case can be investigated where the payoff functions of the players and their interconnections are not initially known to the principal and are to be learned over time.

\bibliographystyle{IEEEtran}        
\bibliography{arr} 

\begin{thebibliography}{10}
\providecommand{\url}[1]{#1}
\csname url@samestyle\endcsname
\providecommand{\newblock}{\relax}
\providecommand{\bibinfo}[2]{#2}
\providecommand{\BIBentrySTDinterwordspacing}{\spaceskip=0pt\relax}
\providecommand{\BIBentryALTinterwordstretchfactor}{4}
\providecommand{\BIBentryALTinterwordspacing}{\spaceskip=\fontdimen2\font plus
\BIBentryALTinterwordstretchfactor\fontdimen3\font minus
  \fontdimen4\font\relax}
\providecommand{\BIBforeignlanguage}[2]{{%
\expandafter\ifx\csname l@#1\endcsname\relax
\typeout{** WARNING: IEEEtran.bst: No hyphenation pattern has been}%
\typeout{** loaded for the language `#1'. Using the pattern for}%
\typeout{** the default language instead.}%
\else
\language=\csname l@#1\endcsname
\fi
#2}}
\providecommand{\BIBdecl}{\relax}
\BIBdecl

\bibitem{27}
C.~Carraro and A.~Sgobbi, ``Modelling negotiated decision making in
  environmental and natural resource management: a multilateral, multiple
  issues, non-cooperative bargaining model with uncertainty,''
  \emph{Automatica}, vol.~44, no.~6, pp. 1488--1503, 2008.

\bibitem{28}
K.~Chan and B.~Rivera, ``Network science for decision-making: Impact of
  distributed information quality on performance of decision-making groups,''
  ARMY RESEARCH LAB ADELPHI MD, Tech. Rep., 2011.

\bibitem{29}
E.~K. Zavadskas and Z.~Turskis, ``Multiple criteria decision making (mcdm)
  methods in economics: an overview,'' \emph{Technological and Economic
  Development of Economy}, vol.~17, no.~2, pp. 397--427, 2011.

\bibitem{31}
J.~C. Harsanyi, ``Games with incomplete information played by {“Bayesian”}
  players, i--iii part i. the basic model,'' \emph{Management science},
  vol.~14, no.~3, pp. 159--182, 1967.

\bibitem{24}
R.~J. Aumann and S.~Hart, \emph{Handbook of Game Theory with Economic
  Applications}.\hskip 1em plus 0.5em minus 0.4em\relax North-Holland
  Amsterdam, 1992, vol.~1.

\bibitem{19}
A.~Gupta, C.~Langbort, and T.~Ba{\c{s}}ar, ``Dynamic games with asymmetric
  information and resource constrained players with applications to security of
  cyberphysical systems,'' \emph{IEEE Transactions on Control of Network
  Systems}, vol.~4, no.~1, pp. 71--81, 2016.

\bibitem{20}
J.~R. Marden, ``The role of information in distributed resource allocation,''
  \emph{IEEE Transactions on Control of Network Systems}, vol.~4, no.~3, pp.
  654--664, 2016.

\bibitem{26}
C.~E. Walsh, ``Optimal economic transparency,'' \emph{Eighth issue (March 2007)
  of the International Journal of Central Banking}, 2018.

\bibitem{1}
C.~Cornand and F.~Heinemann, ``Optimal degree of public information
  dissemination,'' \emph{The Economic Journal}, vol. 118, no. 528, pp.
  718--742, 2008.

\bibitem{32}
F.~Parise, S.~Grammatico, B.~Gentile, and J.~Lygeros, ``Distributed convergence
  to nash equilibria in network and average aggregative games,''
  \emph{Automatica}, vol. 117, p. 108959, 2020.

\bibitem{4}
D.~Acemoglu, A.~Ozdaglar, and A.~Tahbaz-Salehi, ``Networks, shocks, and
  systemic risk,'' National Bureau of Economic Research, Tech. Rep., 2015.

\bibitem{5}
L.~Stella, F.~Bagagiolo, D.~Bauso, and G.~Como, ``Opinion dynamics and
  stubbornness through mean-field games,'' in \emph{52nd IEEE Conference on
  Decision and Control}.\hskip 1em plus 0.5em minus 0.4em\relax IEEE, 2013, pp.
  2519--2524.

\bibitem{6}
T.~Roughgarden, ``Routing games,'' \emph{Algorithmic Game Theory}, vol.~18, pp.
  459--484, 2007.

\bibitem{21}
V.~G. Lopez, Y.~Wan, and F.~L. Lewis, ``Bayesian graphical games for
  synchronization in networks of dynamical systems,'' \emph{IEEE Transactions
  on Control of Network Systems}, vol.~7, no.~2, pp. 1028--1039, 2019.

\bibitem{22}
B.~Gharesifard and S.~L. Smith, ``Distributed submodular maximization with
  limited information,'' \emph{IEEE Transactions on Control of Network
  Systems}, vol.~5, no.~4, pp. 1635--1645, 2017.

\bibitem{23}
M.~Rasouli and D.~Teneketzis, ``An efficient market design for electricity
  networks with strategic users possessing local information,'' \emph{IEEE
  Transactions on Control of Network Systems}, vol.~6, no.~3, pp. 1038--1049,
  2019.

\bibitem{30}
C.~A. Van~der Cruijsen, S.~C. Eijffinger, and L.~H. Hoogduin, ``Optimal central
  bank transparency,'' \emph{Journal of International Money and Finance},
  vol.~29, no.~8, pp. 1482--1507, 2010.

\bibitem{2}
L.~Colombo, G.~Femminis, and A.~Pavan, ``Information acquisition and welfare,''
  \emph{The Review of Economic Studies}, vol.~81, no.~4, pp. 1438--1483, 2014.

\bibitem{3}
G.-M. Angeletos and A.~Pavan, ``Efficient use of information and social value
  of information,'' \emph{Econometrica}, vol.~75, no.~4, pp. 1103--1142, 2007.

\bibitem{10}
A.~Tabarrok, ``The private provision of public goods via dominant assurance
  contracts,'' \emph{Public Choice}, vol.~96, no. 3-4, pp. 345--362, 1998.

\bibitem{11}
R.~D. Roberts, ``Government subsidies to private spending on public goods,''
  \emph{Public Choice}, vol.~74, no.~2, pp. 133--152, 1992.

\bibitem{12}
I.~V. Gashenko, Y.~S. Zima, and A.~V. Davidyan, \emph{Optimization of the
  Taxation System: Preconditions, Tendencies and Perspectives}.\hskip 1em plus
  0.5em minus 0.4em\relax Springer, 2019.

\bibitem{8}
C.~Jernigan and B.~F. Mistree, ``Gaydar: Facebook friendships expose sexual
  orientation,'' \emph{First Monday}, 2009.

\bibitem{9}
A.~Acquisti, C.~Taylor, and L.~Wagman, ``The economics of privacy,''
  \emph{Journal of Economic Literature}, vol.~54, no.~2, pp. 442--92, 2016.

\bibitem{akyol:16}
E.~Akyol, C.~Langbort, and T.~Ba{\c{s}}ar, ``Information-theoretic approach to
  strategic communication as a hierarchical game,'' \emph{Proceedings of the
  IEEE}, vol. 105, no.~2, pp. 205--218, 2016.

\bibitem{sayin:19}
M.~O. Sayin, E.~Akyol, and T.~Ba{\c{s}}ar, ``Hierarchical multistage gaussian
  signaling games in noncooperative communication and control systems,''
  \emph{Automatica}, vol. 107, pp. 9--20, 2019.

\bibitem{akyol2015privacy}
E.~Akyol, C.~Langbort, and T.~Ba{\c{s}}ar, ``Privacy constrained information
  processing,'' in \emph{2015 54th IEEE Conference on Decision and Control
  (CDC)}.\hskip 1em plus 0.5em minus 0.4em\relax IEEE, 2015, pp. 4511--4516.

\bibitem{15}
S.~Bolouki, D.~G. Dobakhshari, T.~Ba{\c{s}}ar, V.~Gupta, and A.~Nedi{\'c},
  ``Applications of group testing to security decision-making in networks,'' in
  \emph{2017 IEEE 56th Annual Conference on Decision and Control (CDC)}.\hskip
  1em plus 0.5em minus 0.4em\relax IEEE, 2017, pp. 2929--2934.

\bibitem{17}
D.~Berend and T.~Tassa, ``Improved bounds on bell numbers and on moments of
  sums of random variables,'' \emph{Probability and Mathematical Statistics},
  vol.~30, no.~2, pp. 185--205, 2010.

\bibitem{16}
S.~Boyd, S.~P. Boyd, and L.~Vandenberghe, \emph{Convex Optimization}.\hskip 1em
  plus 0.5em minus 0.4em\relax Cambridge university press, 2004.

\bibitem{13}
R.~A. Miura-Ko, B.~Yolken, J.~Mitchell, and N.~Bambos, ``Security
  decision-making among interdependent organizations,'' in \emph{2008 21st IEEE
  Computer Security Foundations Symposium}.\hskip 1em plus 0.5em minus
  0.4em\relax IEEE, 2008, pp. 66--80.

\bibitem{14}
Z.~Zhou, N.~Bambos, and P.~Glynn, ``Dynamics on linear influence network games
  under stochastic environments,'' in \emph{International Conference on
  Decision and Game Theory for Security}.\hskip 1em plus 0.5em minus
  0.4em\relax Springer, 2016, pp. 114--126.

\bibitem{18}
A.~Lancichinetti, S.~Fortunato, and J.~Kert{\'e}sz, ``Detecting the overlapping
  and hierarchical community structure in complex networks,'' \emph{New Journal
  of Physics}, vol.~11, no.~3, p. 033015, 2009.

\end{thebibliography}

\end{document}